\documentclass[12pt]{amsart}%
\usepackage{amsfonts}
\usepackage{amsmath}
\usepackage{amssymb}
\usepackage{graphicx}%
\newtheorem{theorem}{Theorem}[section]
\newtheorem{lemma}[theorem]{Lemma}

\newtheorem{proposition}[theorem]{Proposition}

\theoremstyle{definition}
\newtheorem{definition}[theorem]{Definition}
\newtheorem{example}[theorem]{Example}
\theoremstyle{remark}

\numberwithin{equation}{section}

\begin{document}
\title{Distributional point values and delta sequences}

\begin{abstract}
Recently Sasane \cite{Sasane} defined a notion of evaluating a distribution at
a point using delta sequences. In this paper, we explore the relationship
between generalizations of his definition and the standard definition of
distributional point values. This allows us to obtain a description of
distributional point values via delta sequences and a characterization of when
a distribution is actually a regular distribution given by bounded function.
We also give a characterization of limits in a continuous variable by the
existence of the limits of certain sequences.

\end{abstract}
\author{Ricardo Estrada}
\address{Department of Mathematics\\
Louisiana State University\\
Baton Rouge\\
LA 70803\\
USA}
\email{restrada@math.lsu.edu}
\author{Kevin Kellinsky-Gonzalez}
\address{Department of Mathematics\\
Louisiana State University\\
Baton Rouge\\
LA 70803\\
USA}
\email{kkelli2@lsu.edu}
\maketitle

\section{Introduction\label{Section:Introduction}}

Distributional point values were first defined in one variable by
\L ojasiewics \cite{Lojasiewics}. His definition is given as a distributional
limit over a \emph{continuous} variable. In other words, if $f\in
\mathcal{D}^{\prime}\left(  \mathbb{R}\right)  $ and $x_{0}\in\mathbb{R}$ then
we say that $f$ has a distributional point value, equal to $\gamma,$ at
$x_{0}$ if%
\begin{equation}
\lim_{\varepsilon\rightarrow0}f\left(  x_{0}+\varepsilon x\right)  =\gamma\,,
\label{Int 1}%
\end{equation}
in the distributional sense, that is, if for all test functions $\phi
\in\mathcal{D}\left(  \mathbb{R}\right)  $ we have that%
\begin{equation}
\lim_{\varepsilon\rightarrow0}\left\langle f\left(  x_{0}+\varepsilon
x\right)  ,\phi\left(  x\right)  \right\rangle =\gamma\int_{-\infty}^{\infty
}\phi\left(  x\right)  \,\mathrm{d}x\,. \label{Int 2}%
\end{equation}
Similarly \cite{L2} point values in several variables are defined as a
distributional limit over a continuous variable. Point values have been
studied extensively and are the first step in the study of distributional
asymptotic analysis and of the study of local properties of distributions
\cite{CamposFerreira, EK2002, Peetre, pil, PSV, vla}.

It is possible to find in the literature other definitions of distributional
point values, based on the use of \emph{delta sequences.} For instance, in a
recent study, Sasane \cite{Sasane} uses the alternative definition
\textquotedblleft$f\left(  x_{0}\right)  =\eta$\textquotedblright\ if for all
positive and even test functions $\phi$ with $\int_{-\infty}^{\infty}%
\phi\left(  x\right)  \,\mathrm{d}x=1$ one has%
\begin{equation}
\lim_{n\rightarrow\infty}\left\langle f\left(  x_{0}+x\right)  ,\phi
_{n}\left(  x\right)  \right\rangle =\eta\,, \label{Int 4}%
\end{equation}
where $\left\{  \phi_{n}\right\}  _{n=1}^{\infty}$ is the standard delta
sequence generated by $\phi,$ namely, $\phi_{n}\left(  x\right)  =n\phi\left(
nx\right)  .$ Naturally the question arises if the two definitions are
equivalent. More generally, if $\mathfrak{F}$\ is a family of test functions,
we would like to consider the relationship between the existence of the
distributional point value and the existence of the limit (\ref{Int 4})
whenever the sequence $\left\{  \phi_{n}\right\}  _{n=1}^{\infty}$ belongs to
$\mathfrak{F}.$ Interestingly, the two definitions are \emph{not }equivalent
for many classes of delta sequences, in particular for the family considered
in \cite{Sasane}. Nevertheless, we are able to show that for \emph{some}
classes they are actually equivalent.

In order to study this problem, we start by studying a very general question
about limits. Indeed, in a metric space $X,$ given a function $f:X\setminus
\left\{  x_{0}\right\}  \rightarrow\mathbb{R},$ then the limit
\begin{equation}
\lim_{x\rightarrow x_{0}}f\left(  x\right)  =L\,, \label{Int 5}%
\end{equation}
exists if and only if
\begin{equation}
\lim_{n\rightarrow\infty}f\left(  x_{n}\right)  =L\,, \label{Int 6}%
\end{equation}
for all sequences $\left\{  x_{n}\right\}  _{n=1}^{\infty}$ in $X\setminus
\left\{  x_{0}\right\}  $ that converge to $x_{0}.$ The question we would like
to consider is whether the existence of the limit $\lim_{n\rightarrow\infty
}f\left(  x_{n}\right)  $ for sequences $\left\{  x_{n}\right\}
_{n=1}^{\infty}$ of a certain family implies that (\ref{Int 5}) is satisfied.
In Section \ref{Section: The continuous case},\ we consider the case where
$X=(0,\infty],$ $x_{0}=\infty$ and $f$ continuous, showing that in such cases
the existence of the limit%
\begin{equation}
\lim_{n\rightarrow\infty}f\left(  na\right)  =F\left(  a\right)  \,,
\label{Int 7}%
\end{equation}
for all $a>0$ implies that, in fact, $F$ is a constant function, $F\left(
a\right)  =L,$ for all $a>0,$ and that $\lim_{x\rightarrow\infty}f\left(
x\right)  =L.$ We give examples of other families of sequences for which
$\lim_{x\rightarrow\infty}f\left(  x\right)  =L$ might not hold true. We are
also able to present, in Section \ref{Section: The measurable case}, a
corresponding result when the function $f$ in (\ref{Int 7}) is not necessarily
continuous but just measurable and the limit holds almost everywhere.

The plan of the rest of the article is as follows. Basic results on
distributional point values are briefly discussed in Section
\ref{Section:Preliminaries}, while delta sequences are considered in Section
\ref{Section: Deltasequences}. Section \ref{Section: Several lemmas} gives
several useful results on the characterization of distributions and functions
using normalized positive test functions. The main results are given in
Section \ref{Section: Comparison of definitions}, where we give equivalent
conditions to the existence of point values obtained from a given family of
delta sequences. In particular, we prove that the existence of the
distributional point value is equivalent to the existence of the point value
for\ the family of standard delta sequences generated by a positive normalized
test function. Then we show that for the family employed in \cite{Sasane}\ the
equivalence is the existence of the symmetric distributional point value. We
also consider radial delta sequences in several variables, and finish by
studying the family of all delta sequences of normalized positive test functions.

\section{Preliminaries\label{Section:Preliminaries}}

We refer to the texts for the basic ideas about distributions \cite{AMS,
Kanwal, Schwartz, Vladimirovbook}. Ideas on the local behavior of
distributions can be found in \cite{CamposFerreira, EK2002, pil, PSV,
vla}.\ In this article, we will work mainly in the space $\mathcal{D}^{\prime
}\left(  \mathbb{R}^{d}\right)  $ of distributions on $\mathbb{R}^{d},$ dual
of the space $\mathcal{D}\left(  \mathbb{R}^{d}\right)  $ of standard test
functions, that is, $C^{\infty}$ functions with compact support, with its
inductive limit topology \cite{Treves}.

If $f\in\mathcal{D}^{\prime}\left(  \mathbb{R}\right)  $ and $x_{0}%
\in\mathbb{R}$ then \cite{Lojasiewics}\ we say that $f$ has a
\emph{distributional point value,} equal to $\gamma,$ at $x_{0}$ if
$\lim_{\varepsilon\rightarrow0}f\left(  x_{0}+\varepsilon x\right)  =\gamma,$
in the strong topology of $\mathcal{D}^{\prime}\left(  \mathbb{R}\right)  .$
Equivalently, since a sequence of distributions converges strongly if and only
if it converges weakly \cite{Treves}, if for all test functions $\phi
\in\mathcal{D}\left(  \mathbb{R}\right)  $ we have that%
\begin{equation}
\lim_{\varepsilon\rightarrow0}\left\langle f\left(  x_{0}+\varepsilon
x\right)  ,\phi\left(  x\right)  \right\rangle =\gamma\int_{-\infty}^{\infty
}\phi\left(  x\right)  \,\mathrm{d}x\,. \label{Pre 2}%
\end{equation}

Interestingly, the existence of the distributional limit $\lim_{\varepsilon
\rightarrow0}f\left(  x_{0}+\varepsilon x\right)  $ implies that this limit is
a constant and that the point value exists. On the other hand, if the limit
\begin{equation}
\lim_{\varepsilon\rightarrow0^{+}}f\left(  x_{0}+\varepsilon x\right)
=g\left(  x\right)  \,,\label{Pre 3}%
\end{equation}
exists, then $g$ does not have to be a constant, but it will have the jump
behavior \cite{VEjump}, that is, $g$ is of the form%
\begin{equation}
g\left(  x\right)  =\gamma_{-}H\left(  -x\right)  +\gamma_{+}H\left(
x\right)  \,,\label{Pre 4}%
\end{equation}
where $H$ is the Heaviside function and $\gamma_{\pm}$ are some constants.
Distributions of the form (\ref{Pre 4}) are the most general homogeneous
distributions of degree $0$ in one variable. Alternatively, (\ref{Pre 3}) and
(\ref{Pre 4}) hold if and if the lateral limits $f\left(  x_{0}\pm0\right)
=\lim_{\varepsilon\rightarrow0^{+}}f\left(  x_{0}\pm\varepsilon x\right)
=\gamma_{\pm},$ exist in $\mathcal{D}^{\prime}\left(  0,\infty\right)  $ and
$f$ does not have delta functions at $x_{0}.$

In several variables point values are defined similarly \cite{L2}, namely, if
$f\in\mathcal{D}^{\prime}\left(  \mathbb{R}^{d}\right)  ,$ then the
distributional point value $f\left(  \mathbf{x}_{0}\right)  $ exists and
equals $\gamma$\ if $\lim_{\varepsilon\rightarrow0}f\left(  \mathbf{x}%
_{0}+\varepsilon\mathbf{x}\right)  =\gamma,$ distributionally. In several
variables the limit $\lim_{\varepsilon\rightarrow0}f\left(  \mathbf{x}%
_{0}+\varepsilon\mathbf{x}\right)  $ could exist without being a constant. In
fact, if%
\begin{equation}
\lim_{\varepsilon\rightarrow0^{+}}f\left(  \mathbf{x}_{0}+\varepsilon
\mathbf{x}\right)  =g\left(  \mathbf{x}\right)  \,,\label{Pre 5}%
\end{equation}
then $g$ is homogeneous of degree $0.$ Homogeneous distributions of degree
zero are given by a formula of the type%
\begin{equation}
\left\langle g\left(  \mathbf{x}\right)  ,\phi\left(  \mathbf{x}\right)
\right\rangle =\int_{0}^{\infty}\left\langle \alpha\left(  \mathbf{w}\right)
,\phi\left(  r\mathbf{w}\right)  \right\rangle _{\mathcal{D}^{\prime}\left(
\mathbb{S}\right)  \times\mathcal{D}\left(  \mathbb{S}\right)  }%
r^{d-1}\,\mathrm{d}r\,,\label{Pre 6}%
\end{equation}
for a certain distribution $\alpha\in\mathcal{D}^{\prime}\left(
\mathbb{S}\right)  $ \cite[Thm. 2.6.2]{EK2002}. The distribution $\alpha$ is
the \emph{thick distributional value} \cite{Estrada12} of $f$ at
$\mathbf{x}_{0},$ namely, $f$ has no delta functions at $\mathbf{x}_{0}$ and
$\alpha$ is the thick limit%
\begin{equation}
\lim_{\varepsilon\rightarrow0^{+}}f\left(  \mathbf{x}_{0}+r\varepsilon
\mathbf{w}\right)  =\alpha\left(  \mathbf{w}\right)  \,,\label{Pre 7}%
\end{equation}
in the space $\mathcal{D}^{\prime}\left(  \left(  0,\infty\right)
,\mathcal{D}^{\prime}\left(  \mathbb{S}\right)  \right)  ,$ that is, for all
$\rho\in\mathcal{D}\left(  0,\infty\right)  ,$%
\begin{equation}
\left\langle \lim_{\varepsilon\rightarrow0^{+}}f\left(  \mathbf{x}%
_{0}+r\varepsilon\mathbf{w}\right)  ,\rho\left(  r\right)  \right\rangle
_{\mathcal{D}^{\prime}\left(  0,\infty\right)  \times\mathcal{D}\left(
0,\infty\right)  }=\left(  \int_{0}^{\infty}\rho\left(  r\right)
\,\mathrm{d}r\right)  \alpha\left(  \mathbf{w}\right)  \,,\label{Pre 8}%
\end{equation}

\section{The continuous case\label{Section: The continuous case}}

We start with a general known result that will be useful in our analysis.\smallskip

\begin{proposition}
\label{Prop. DD 1}Let $f:\left(  0,\infty\right)  \rightarrow\mathbb{R}$ be
continuous. Suppose that for each $a>0$ the sequence $\left\{  f\left(
an\right)  \right\}  _{n=1}^{\infty}$ converges, to $F\left(  a\right)  .$
Then $F\left(  a\right)  $ does not depend on $a,$ that is,%
\begin{equation}
F\left(  a\right)  =L\text{ for all }a>0\,, \label{DD 1}%
\end{equation}
for some $L,$ and actually%
\begin{equation}
\lim_{x\rightarrow\infty}f\left(  x\right)  =L\,. \label{DD 2}%
\end{equation}

\end{proposition}

\begin{proof}
Clearly the function $F$ is constant in each class of the quotient space
$\mathbb{R}/\mathbb{Q},$ $F\left(  ra\right)  =F\left(  a\right)  $ if
$r\in\mathbb{Q}.$ Also, $F$ is continuous or of the first Baire class
\cite{Gordon, Natanson}, so that the set $D$ of points of continuity of $F$ is
dense in $\left(  0,\infty\right)  .$ Let $\alpha\in D.$ Let $b>0.$ If
$\varepsilon>0$ then there exists $\delta>0$ such that $\left\vert
a-\alpha\right\vert <\delta$ implies $\left\vert F\left(  a\right)  -F\left(
\alpha\right)  \right\vert <\varepsilon$ and there exist $r$ rational such
that $\left\vert rb-\alpha\right\vert <\delta.$ Therefore
\begin{equation}
\left\vert F\left(  b\right)  -F\left(  \alpha\right)  \right\vert =\left\vert
F\left(  rb\right)  -F\left(  \alpha\right)  \right\vert <\varepsilon\,,
\label{DD 3}%
\end{equation}
and since $\varepsilon$ is arbitrary, $F\left(  b\right)  =F\left(
\alpha\right)  .$

In order to prove (\ref{DD 2}), observe that if $\varepsilon>0,$ then for each
$a\in\left[  1,2\right]  $ there exists $n_{0}=n_{0}\left(  a\right)  $ such
that $\left\vert f\left(  ka\right)  -L\right\vert <\varepsilon$ for $k\geq
n_{0}\left(  a\right)  .$ This means that%
\begin{equation}
\left[  1,2\right]  =\bigcup_{n=1}^{\infty}\bigcap_{k\geq n}\left\{
a\in\left[  1,2\right]  :\left\vert f\left(  ka\right)  -L\right\vert
<\varepsilon\right\}  \,. \label{DD 4}%
\end{equation}
Therefore there exists $n_{0}$ such that $\bigcap_{k\geq n_{0}}\left\{
a\in\left[  1,2\right]  :\left\vert f\left(  ka\right)  -L\right\vert
<\varepsilon\right\}  $ contains an interval $I=\left[  \alpha,\beta\right]  $
with $\alpha\neq\beta.$ Observe now that $\bigcup_{k=n_{0}}^{\infty}kI$
contains a ray $(B,\infty),$ since in fact $\bigcup_{k=n_{1}}^{\infty}kI$ is a
closed ray if $n_{1}>1/\left(  \beta-\alpha\right)  .$ Hence if $x>B$ then
$x=ka$ for some $k\geq n_{0}$ and some $a\in I$ and, consequently, $\left\vert
f\left(  x\right)  -L\right\vert =\left\vert f\left(  ka\right)  -L\right\vert
<\varepsilon.$\smallskip
\end{proof}

It is interesting that there are sequences $\left\{  \xi_{n}\right\}
_{n=1}^{\infty}$ with $\lim_{n\rightarrow\infty}\xi_{n}=\infty$ such that for
some continuous functions $f:\left(  0,\infty\right)  \rightarrow\mathbb{R}$
the limit%
\begin{equation}
\lim_{n\rightarrow\infty}f\left(  a\xi_{n}\right)  =G\left(  a\right)  
\label{DDn 1}%
\end{equation}
exists for all $a>0$, but the function $G$ is not constant. Indeed, let
$f\left(  x\right)  =\sin\left(  2\pi\ln x\right)$ and $\xi_{n}%
=e^{(n+1/n)}.$ The limit $\lim_{n\rightarrow\infty}f\left(  a\xi_{n}\right)
=\sin\left(  2\pi\ln a\right)  ,$ exists but it is not constant, of course.

\section{The measurable case\label{Section: The measurable case}}

We shall now consider an extension of the results of Section
\ref{Section: The continuous case} to measurable functions.\smallskip

\begin{proposition}
\label{Lemma Meas 1}Suppose $f:(0,\infty)\rightarrow\mathbb{R}$ is measurable.
For $a>0$, suppose that the function $F$ defined by%
\begin{equation}
F(a)=\lim\limits_{n\rightarrow\infty}f(an)\,, \label{meas 1}%
\end{equation}
is well-defined almost everywhere. Then $F$ is constant almost everywhere.
\end{proposition}

\begin{proof}
Let us first suppose that $f\in L^{\infty}(0,\infty).$ Let $\phi\in
\mathcal{D}(0,\infty)$ be a test function. For $\lambda>0,$ let us define%
\begin{equation}
G(\lambda)=\int_{0}^{\infty}f(\lambda x)\phi(x)\,\mathrm{d}x\,. \label{meas 2}%
\end{equation}

The function $G$ is continuous because $\phi\in\mathcal{D}(0,\infty)$. For a
fixed $\lambda,$ let us consider the sequence $\left\{  G(\lambda n)\right\}
_{n=1}^{\infty}.$ Since $f$ is bounded, we can apply the dominated convergence
theorem to see that
\begin{equation}
\lim\limits_{n\rightarrow\infty}G(\lambda n)=\int_{0}^{\infty}\lim
\limits_{n\rightarrow\infty}f(\lambda nx)\phi(x)\,\mathrm{d}x=\int_{0}%
^{\infty}F(\lambda x)\phi(x)\,\mathrm{d}x\,,\nonumber
\end{equation}
exists. The Proposition \ref{Prop. DD 1}\ then yields that $\int_{0}^{\infty
}F(\lambda x)\phi(x)\,\mathrm{d}x$ does not depend on $\lambda,$%
\begin{equation}
\int_{0}^{\infty}F(\lambda x)\phi(x)\,\mathrm{d}x=\int_{0}^{\infty}%
F(x)\phi(x)\,\mathrm{d}x\,. \label{meas 4}%
\end{equation}
Therefore, the regular distribution $F$ is constant since $F(\lambda x)=F(x),$
$\lambda>0,$ and only the constants are homogeneous of degree $0$ in the
interval $\left(  0,\infty\right)  $ \cite{EK2002}, that is, $F(x)=C,$ as
\emph{distributions.} Notice now that the locally integrable function that
gives a regular distribution is unique almost everywhere, so that
$F(x)=C\ \ \ \left(  \text{a.e.}\right)  .$

Let us now consider the case of a general measurable function $f$ for which
the limit $\lim\limits_{n\rightarrow\infty}f(an)=F(a)$ exists (a.e.). We can
then define the bounded function%
\begin{equation}
h(x)=\arctan f{(x)}\,. \label{meas 7}%
\end{equation}
Then $\lim\limits_{n\rightarrow\infty}h{(}an{)}=\arctan F(a)$ exists (a.e.).
Consequently, $\arctan F(a)$ is constant, and therefore so is $F(a).$%
\smallskip
\end{proof}

In Proposition \ref{Prop. DD 1}, it is shown that in the continuous case not
only is $F\left(  a\right)  =L$ for all $a>0,$ but actually $\lim
_{x\rightarrow\infty}f\left(  x\right)  =L.$ This is no longer true in the
measurable case; for example, if $f=\chi_{B},$ the characteristic function of
a set $B$ of measure zero such that $B\cap\left(  x,\infty\right)
\neq\emptyset$ for all $x>0,$ then $\lim_{x\rightarrow\infty}f\left(
x\right)  $ does not exist. Of course, this function $f$ is equal almost
everywhere to a function $\widetilde{f},$ the zero function, for which
$\lim_{x\rightarrow\infty}\widetilde{f}\left(  x\right)  $ exists. An example
where (\ref{meas 1}) exists for all $a>0$ but $\lim_{x\rightarrow\infty
}\widetilde{f}\left(  x\right)  $ does not exist for any function $f$ such
that $f\left(  x\right)  =\widetilde{f}\left(  x\right)  $ (a.e.) can be
constructed as follows. The strategy will be to construct an unbounded set, $A$, with measure $0$ such that $f\left(x\right) \neq 0$ for infinitely many $x\in A$ \smallskip

\begin{example}
\label{Example 1}Let $\left\{  N_{k}\right\}  _{k=1}^{\infty}$ be a sequence
of positive integers such that%
\begin{equation}
kN_{k}<N_{k+1}\,. \label{ex 1}%
\end{equation}
For each $k$ let us choose a \ non empty open interval $B_{k}\subset\left(
N_{k}-1,N_{k}\right)  .$ Then if $j\in\mathbb{N},$ $j\left(  \frac{1}%
{k},1\right)  \cap B_{k}\neq\emptyset$ only if $N_{k}\leq j<N_{k+1}.$
Therefore, if $x\in\left(  \frac{1}{k},1\right)  ,$ then%
\begin{equation}
\chi_{B_{k}}\left(  jx\right)  =0\,\ \text{for all \ }j\in\mathbb{N}%
\,,\ \ x\notin A_{k}\,, \label{ex 3}%
\end{equation}
where $A_{k}=\bigcup_{j=N_{k}}^{N_{k+1}-1}\frac{1}{j}B_{k}\,.$ Let now
$\left\{  \eta_{k}\right\}  _{k=1}^{\infty}$ be a sequence of strictly
positive numbers such that the series $\sum_{k=1}^{\infty}\eta_{k}$ converges
and let us further restrict the sets $B_{k}$ by requiring that $\mu\left(
A_{k}\right)  <\eta_{k},$ for all $k.$ Let $A=\limsup_{k\rightarrow\infty
}A_{k}=\bigcap_{k=1}^{\infty}\bigcup_{q=k}^{\infty}A_{q}.$ Then $\mu\left(
A\right)  =0.$

Let us now define the function $f:(0,\infty)\rightarrow\mathbb{R}$ by%
\begin{equation}
f\left(  x\right)  =\sum_{k=1}^{\infty}\chi_{B_{k}}\left(  x\right)  \,.
\label{ex 6}%
\end{equation}
If $\widetilde{f}\left(  x\right)  =f\left(  x\right)  $ almost everywhere,
then the limit $\lim_{x\rightarrow\infty}\widetilde{f}\left(  x\right)  $ does
not exist. On the other hand, $\lim_{n\rightarrow\infty}f\left(  nx\right)  $
exists almost everywhere. Indeed, it is enough to show the existence almost
everywhere in $\left(  0,1\right)  ,$ and\ the limit of $f\left(  nx\right)  $
exists\ and equals $0$ if $x\in\left(  0,1\right)  \setminus A$ since if
$x\notin A$ then there exists $k_{0}$ such that $x\notin A_{k}$ for $k\geq
k_{0}$ and consequently, $f\left(  nx\right)  =0$ whenever $n\geq N_{k_{0}}%
.$\smallskip
\end{example}

An example involving continuous functions can be obtained by a slight
modification.\smallskip

\begin{example}
\label{Example 2}Let $g$ be a continuous function in $\left(  0,\infty\right)
$ such that $0\leq g\left(  x\right)  \leq f\left(  x\right)  $ and such that
there exist points $\xi_{k}\in B_{k},$ for all $k,$ such that $g\left(
\xi_{k}\right)  =1.$ Then $\lim_{n\rightarrow\infty}g\left(  nx\right)  =0$
almost everywhere, \emph{but not everywhere}, since $\lim_{x\rightarrow\infty
}g\left(  x\right)  $ does not exist.\smallskip
\end{example}

The examples show that it is possible for $\lim\limits_{n\rightarrow\infty
}f(an)$ to be equal to a constant $L$ almost everywhere but without
$\lim_{x\rightarrow\infty}f\left(  x\right)  $ existing. We do have a
convergence in measure type result.\smallskip

\begin{proposition}
\label{Prop. DD 2}Suppose $f:(0,\infty)\rightarrow\mathbb{R}$ is measurable.
Suppose that%
\begin{equation}
\lim\limits_{n\rightarrow\infty}f(an)=L\,\ \ \ \left(  \text{a.e.}\right)  \,.
\label{meas 9}%
\end{equation}
Then for all $\varepsilon>0$ and all $C>1,$%
\begin{equation}
\lim_{x\rightarrow\infty}\frac{\mu\left(  \left\{  t\in\left[  x,Cx\right]
:\left\vert f\left(  t\right)  -L\right\vert >\varepsilon\right\}  \right)
}{\mu\left(  \left[  x,Cx\right]  \right)  }=0\,, \label{meas 10}%
\end{equation}
where $\mu$ denotes the Lebesgue measure of a set.
\end{proposition}

\begin{proof}
Let us denote by $G\left(  x\right)  $ the quotient $\mu\left(  \left\{
t\in\left[  x,cx\right]  :\left\vert f\left(  t\right)  -L\right\vert
>\varepsilon\right\}  \right)  /\left(  1-C\right)  x.$ Notice that $G$ is a
continuous function in $\left(  0,\infty\right)  .$ Let $a>0$ be fixed and
consider the sequence of functions $f_{n}\left(  x\right)  =f\left(
nx\right)  $ in the interval $\left[  a,Ca\right]  .$ Since $f_{n}$ converges
to $L$ almost everywhere in this finite interval, it converges to $L$ in
measure. This means that for all $\varepsilon>0$ the measure of the set
$\{s\in\left[  a,Ca\right]  :\left\vert f_{n}\left(  s\right)  -L\right\vert
>\varepsilon\}$ tends to zero. But the transformation $t=ns$ gives%
\begin{align*}
\frac{\mu\left(  \{s\in\left[  a,Ca\right]  :\left\vert f_{n}\left(  s\right)
-L\right\vert >\varepsilon\}\right)  }{\mu\left(  \left[  a,Ca\right]
\right)  } &  =\frac{\mu\left(  \{s\in\left[  a,Ca\right]  :\left\vert
f\left(  ns\right)  -L\right\vert >\varepsilon\}\right)  }{\mu\left(  \left[
a,Ca\right]  \right)  }\\
&  =\frac{\mu\left(  \left\{  t\in\left[  na,Cna\right]  :\left\vert f\left(
t\right)  -L\right\vert >\varepsilon\right\}  \right)  }{\mu\left(  \left[
na,Cna\right]  \right)  }\\
&  =G\left(  na\right)  \,,
\end{align*}
so that $\lim_{n\rightarrow\infty}G\left(  na\right)  =0.$ Proposition
\ref{Prop. DD 1} then yields (\ref{meas 10}).
\end{proof}

\section{Delta sequences\label{Section: Deltasequences}}

A sequence $\left\{  f_{n}\right\}  _{n=1}^{\infty}$ of distributions is
called a delta sequence if $f_{n}\left(  \mathbf{x}\right)  \rightarrow
\delta\left(  \mathbf{x}\right)  $ in either the strong \ or the weak topology
of $\mathcal{D}^{\prime}\left(  \mathbb{R}^{d}\right)  ,$ since the two
notions are equivalent \cite{Treves}. In other words, $\left\{  f_{n}\right\}
_{n=1}^{\infty}$ is a delta sequence if%
\begin{equation}
\lim_{n\rightarrow\infty}\left\langle f,\phi\right\rangle =\phi\left(
\mathbf{0}\right)  \,, \label{del 1}%
\end{equation}
for all $\phi\in\mathcal{D}\left(  \mathbb{R}^{d}\right)  .$ In this article
we will be interested mostly in the case when the distributions $f_{n}$ are
actually smooth functions, but general delta sequences are also of interest,
of course. They have been employed in several problems \cite{AMS}, such as the
definitions of point values \cite{Sasane}\ or the definition of products of
distributions \cite{KohLi, Li, Ozcag}.

There are many ways to construct delta sequences. A simple one is the
following. Let $f$ be a fixed distribution of rapid decay at infinity, that
is, $f\in\mathcal{K}^{\prime}\left(  \mathbb{R}^{d}\right)  .$ Then all the
moments $\mu_{\mathbf{k}}=\left\langle f\left(  \mathbf{x}\right)
,\mathbf{x}^{\mathbf{k}}\right\rangle ,$ exist for $\mathbf{k}\in
\mathbb{N}^{d}$ since all polynomials belong to $\mathcal{K}\left(
\mathbb{R}^{d}\right)  ,$ and the \emph{moment asymptotic expansion }%
\begin{equation}
f\left(  \lambda\mathbf{x}\right)  \sim\sum_{q=0}^{\infty}\sum_{\left\vert
\mathbf{k}\right\vert =q}\frac{\mu_{\mathbf{k}}\mathbf{D}^{\mathbf{k}}%
\delta\left(  \mathbf{x}\right)  }{\mathbf{k}!}\frac{1}{\lambda^{q+d}%
}\,,\ \ \text{\ as }\lambda\rightarrow\infty\,, \label{del 3}%
\end{equation}
holds in $\mathcal{K}^{\prime}\left(  \mathbb{R}^{d}\right)  $ \cite{EK2002}.
Therefore, when $\mu_{\mathbf{0}}\neq0$\ if $\left\{  \xi_{n}\right\}
_{n=1}^{\infty}$ is any sequence of positive numbers with $\lim_{n\rightarrow
\infty}\xi_{n}=\infty$ then
\begin{equation}
g_{n}\left(  \mathbf{x}\right)  =\frac{\xi_{n}^{d}}{\mu_{\mathbf{0}}}f\left(
\xi_{n}\mathbf{x}\right)  \,, \label{del 4}%
\end{equation}
is a delta sequence, generated by $f$ and $\left\{  \xi_{n}\right\}
_{n=1}^{\infty}.$ When $\xi_{n}=n$ for all $n,$ we call this sequence the
\emph{standard} delta sequence generated by $f.$

Another useful construction of delta sequences is provided by the ensuing well
known result.\smallskip

\begin{lemma}
\label{Lemma DS 1}Suppose $\left\{  \psi_{n}\right\}  _{n=1}^{\infty}$ is a
sequence of normalized positive test functions in $\mathcal{D}^{\prime}\left(
\mathbb{R}^{d}\right)  $ such that $\operatorname*{supp}\psi_{n}%
\subset\{\mathbf{x:}\left\vert \mathbf{x}\right\vert <r_{n}\},$ where
$\lim_{n\rightarrow\infty}r_{n}=0.$ Then $\left\{  \psi_{n}\right\}
_{n=1}^{\infty}$ is a delta sequence.
\end{lemma}

\begin{proof}
Let $\phi$ be any test function. Then by the first mean value theorem for
integrals,%
\begin{equation}
\left\langle \psi_{n},\phi\right\rangle =\int_{\operatorname*{supp}\psi_{n}%
}\psi_{n}\left(  \mathbf{x}\right)  \phi\left(  \mathbf{x}\right)
\,\mathrm{d}\mathbf{x}=\phi\left(  \mathbf{x}_{n}\right)  \,, \label{DS 1}%
\end{equation}
for some $\mathbf{x}_{n}\in\operatorname*{supp}\psi_{n}.$ Since $\left\vert
\mathbf{x}_{n}\right\vert \leq r_{n}\rightarrow0,$ we obtain that
$\mathbf{x}_{n}\rightarrow0,$ and consequently, $\phi\left(  \mathbf{x}%
_{n}\right)  \rightarrow\phi\left(  \mathbf{0}\right)  .$ Thus $\psi
_{n}\left(  \mathbf{x}\right)  \rightarrow\delta\left(  \mathbf{x}\right)
.$\smallskip
\end{proof}

We now give a notion of point value of a distribution based on delta
sequences. Our definition applies to several spaces of distributions, but the
cases $\mathcal{A}=\mathcal{D},$ $\mathcal{E},$ or $\mathcal{S}$ seem the most
relevant.\smallskip

\begin{definition}
Let $\mathcal{A}\left(  \mathbb{R}^{d}\right)  $ be a space of test functions.
Let $\mathfrak{F}$ be a family of delta sequences whose elements belong to
$\mathcal{A}\left(  \mathbb{R}^{d}\right)  .$\ If $f\in\mathcal{A}^{\prime
}\left(  \mathbb{R}^{d}\right)  $ and $\mathbf{x}_{0}\in\mathbb{R}^{d}$ we say
that the value $f\left(  \mathbf{x}_{0}\right)  $ exists and equals $\gamma$
with respect to $\mathfrak{F}$ if
\begin{equation}
\lim_{n\rightarrow\infty}\left\langle f\left(  \mathbf{x}_{0}+\mathbf{x}%
\right)  ,\phi_{n}\right\rangle =\gamma\,, \label{del 5}%
\end{equation}
for all $\left\{  \phi_{n}\right\}  _{n=1}^{\infty}\in\mathfrak{F}.$ When this
holds we write%
\begin{equation}
f\left(  \mathbf{x}_{0}\right)  =\gamma\ \ \ \left(  \mathfrak{F}\right)  \,.
\label{del 6}%
\end{equation}

\end{definition}

The definition of point value employed by Sasane \cite{Sasane} corresponds to
the case when $d=1,$ $f\in\mathcal{D}^{\prime}\left(  \mathbb{R}\right)  ,$
and $\mathfrak{F}$ is the family of all delta sequences whose elements are the
standard sequences generated from a positive, normalized, and symmetric test
function of $\mathcal{D}\left(  \mathbb{R}\right)  .$

\section{Several lemmas\label{Section: Several lemmas}}

In this section, we present several results on how positive test functions
allow us to study many properties of distributions. In particular, we see how
positive test functions tell us if a distribution is a regular distribution
given by a bounded measurable function and give us the essential supremum and
infimum of such a function. In this section, and only in this section, we will
make a notational difference between a regular distribution $\mathsf{f}%
\in\mathcal{D}^{\prime}\left(  \mathbb{R}^{d}\right)  $ and the locally
integrable function $f\ $that generates it as%
\begin{equation}
\left\langle \mathsf{f}\left(  \mathbf{x}\right)  ,\phi\left(  \mathbf{x}%
\right)  \right\rangle =\int_{\mathbb{R}^{d}}f\left(  \mathbf{x}\right)
\phi\left(  \mathbf{x}\right)  \,\mathrm{d}\mathbf{x}\,,\ \ \ \ \phi
\in\mathcal{D}\left(  \mathbb{R}^{d}\right)  \,. \label{SLe 0}%
\end{equation}
In the rest of the article, we will use the same notation, $f,$ for the
distribution and the function.

Let us start with following simple result.\smallskip

\begin{lemma}
\label{Lemma DD 1}The set of functions of the form%
\begin{equation}
\phi=c_{1}\psi_{1}-c_{2}\psi_{2}\,, \label{DD 5}%
\end{equation}
where $c_{1}$ and $c_{2}$ are constants and where $\psi_{1}$ and $\psi_{2}$
are normalized positive test functions is the whole space $\mathcal{D}\left(
\mathbb{R}^{d}\right)  .$

When $d=1,$ the corresponding space with $\psi_{1}$ and $\psi_{2}$ normalized
positive symmetric test functions is the space of all even test functions.
\end{lemma}

\begin{proof}
It is enough to show that the real valued elements of $\mathcal{D}\left(
\mathbb{R}^{d}\right)  $ have the form (\ref{DD 5}) for some positive
constants $c_{1}$ and $c_{2}.$ Let $\zeta_{1}\in\mathcal{D}\left(
\mathbb{R}^{d}\right)  $ be such that $\zeta_{1}\left(  \mathbf{x}\right)
\geq\max\left\{  \phi\left(  \mathbf{x}\right)  ,0\right\}  $ and let
$\zeta_{2}=\zeta_{1}-\phi.$ Then we write $\zeta_{j}=c_{j}\psi_{j}$ where the
$\psi_{j}$ are normalized positive test functions and $c_{j}=\int
_{\mathbb{R}^{d}}\zeta_{j}\left(  \mathbf{x}\right)  \,\mathrm{d}\mathbf{x}.$
In the symmetric case we just also ask $\zeta_{1}$ to be even.\smallskip
\end{proof}

Our first characterization using positive normalized test functions is the
following.\smallskip

\begin{lemma}
\label{SL 1}Let $\mathsf{f}\in\mathcal{D}^{\prime}\left(  \mathbb{R}%
^{d}\right)  .$ Then $\mathsf{f}$ is a regular distribution in an open set
$U\subset\mathbb{R}^{d},$ given by a bounded function $f\in L^{\infty}\left(
U\right)  $ if and only if there exists a constant $M>0$ such that for all
positive, normalized test functions $\phi\in\mathcal{D}\left(  U\right)  $ we
have%
\begin{equation}
\left\vert \left\langle \mathsf{f}\left(  \mathbf{x}\right)  ,\phi\left(
\mathbf{x}\right)  \right\rangle \right\vert \leq M\,. \label{SLe 1}%
\end{equation}

\end{lemma}

\begin{proof}
If $f\in L^{\infty}\left(  U\right)  .$ Then when $\phi\in\mathcal{D}\left(
U\right)  ,$%
\[
\left\vert \left\langle \mathsf{f}\left(  \mathbf{x}\right)  ,\phi\left(
\mathbf{x}\right)  \right\rangle \right\vert =\left\vert \int_{U}f\left(
\mathbf{x}\right)  \phi\left(  \mathbf{x}\right)  \,\mathrm{d}\mathbf{x}%
\right\vert \leq\left\Vert f\right\Vert _{L^{\infty}\left(  U\right)
}\left\Vert \phi\right\Vert _{L^{1}\left(  U\right)  }\,,
\]
so that if $\phi$ is normalized, $\left\vert \left\langle \mathsf{f}\left(
x\right)  ,\phi\left(  x\right)  \right\rangle \right\vert \leq\left\Vert
f\right\Vert _{L^{\infty}\left(  U\right)  }.$ Therefore (\ref{SLe 1}) holds
with $M=\left\Vert f\right\Vert _{L^{\infty}\left(  U\right)  }.$

Conversely, if (\ref{SLe 1}) is satisfied for some $M>0$ for all positive,
normalized test functions of $U$ then $\left\vert \left\langle \mathsf{f}%
\left(  x\right)  ,\psi\left(  x\right)  \right\rangle \right\vert
\leq2M\left\Vert \psi\right\Vert _{L^{1}\left(  U\right)  }$ for all real test
functions $\psi\in\mathcal{D}\left(  U\right)  ,$ because of Lemma
\ref{Lemma DD 1} (or $4M$ if complex). This means that $\mathsf{f}$ is
continuous in $\mathcal{D}\left(  U\right)  ,$ a dense subspace of
$L^{1}\left(  U\right)  $ with the topology induced by $L^{1}\left(  U\right)
$ in its subspace. Hence, $\mathsf{f}$ admits an extension $f\in\left(
L^{1}\left(  U\right)  \right)  ^{\prime}\simeq L^{\infty}\left(  U\right)  ,$
and this means that%
\begin{equation}
\left\langle \mathsf{f}\left(  \mathbf{x}\right)  ,\psi\left(  \mathbf{x}%
\right)  \right\rangle =\int_{U}f\left(  \mathbf{x}\right)  \psi\left(
\mathbf{x}\right)  \,\mathrm{d}\mathbf{x}\,, \label{SLe 3}%
\end{equation}
for all $\psi\in\mathcal{D}\left(  U\right)  .$ Therefore, $\mathsf{f}$ is a
regular distribution given by the bounded function $f$ in the open set
$U.$\smallskip
\end{proof}

In the proof we can see that $\inf\left\{  M:\text{ (\ref{SLe 1})
holds}\right\}  \leq\left\Vert f\right\Vert _{L^{\infty}\left(  U\right)  }.$
In fact, we have more.\smallskip

\begin{lemma}
\label{Lemma SL 2}If $\mathsf{f}\in\mathcal{D}^{\prime}\left(  \mathbb{R}%
^{d}\right)  $ is a regular distribution in $U,$ given by a bounded function
$f\in L^{\infty}\left(  U\right)  $ then%
\begin{equation}
\left\Vert f\right\Vert _{L^{\infty}\left(  U\right)  }=\inf\left\{  M:\text{
(\ref{SLe 1}) holds for all positive, normalized test functions}\right\}  \,,
\label{SLe 2}%
\end{equation}
and%
\begin{equation}
\left\Vert f\right\Vert _{L^{\infty}\left(  U\right)  }=\sup\left\{
\left\vert \left\langle \mathsf{f},\phi\right\rangle \right\vert :\text{ }%
\phi\in\mathcal{D}\left(  U\right)  \text{ positive, normalized test
function}\right\}  \,. \label{SLe 4}%
\end{equation}

\end{lemma}

\begin{proof}
Clearly $\inf\left\{  M:\text{ (\ref{SLe 1}) holds for all positive,
normalized test functions}\right\}  $ is equal to $\sup\left\{  \left\vert
\left\langle \mathsf{f},\phi\right\rangle \right\vert :\text{ }\phi
\in\mathcal{D}\left(  U\right)  \text{ positive, normalized test
function}\right\}  ;$ let us call this $K.$ We know that $K\leq\left\Vert
f\right\Vert _{L^{\infty}\left(  U\right)  }.$ To prove the converse
inequality, let $s<\left\Vert f\right\Vert _{L^{\infty}\left(  U\right)  }.$
Then there exists $\mathbf{x}_{0}\in U$ such that the distributional point
value $\mathsf{f}\left(  \mathbf{x}_{0}\right)  $ exists and $s<\left\vert
\mathsf{f}\left(  \mathbf{x}_{0}\right)  \right\vert .$ If $\phi$ is a
positive normalized test function, then so are the test functions
$\varphi_{\lambda}\left(  \mathbf{x}\right)  =\lambda^{d}\phi\left(
\mathbf{x}_{0}+\lambda\mathbf{x}\right)  $ for all $\lambda>0,$ and if
$\lambda$ is big enough, $\varphi_{\lambda}\in\mathcal{D}\left(  U\right)  .$
Since $\lim_{\lambda\rightarrow\infty}\left\langle \mathsf{f},\varphi
_{\lambda}\right\rangle =\mathsf{f}\left(  \mathbf{x}_{0}\right)  ,$ we can
find $\lambda$ such that $\left\vert \left\langle \mathsf{f},\varphi_{\lambda
}\right\rangle \right\vert >s.$ Consequently, $K>s,$ and because $s<\left\Vert
f\right\Vert _{L^{\infty}\left(  U\right)  }$ is arbitrary, $K\geq\left\Vert
f\right\Vert _{L^{\infty}\left(  U\right)  }.$\smallskip
\end{proof}

In fact, the same argument in the proof of Lemma \ref{Lemma SL 2} allows us to
obtain the ensuing.\smallskip

\begin{lemma}
\label{Lemma SL 3}If $\mathsf{f}\in\mathcal{D}^{\prime}\left(  \mathbb{R}%
^{d}\right)  $ is a real regular distribution in $U,$ given by a function
$f\in L^{1}\left(  U\right)  $ then the essential supremum and infimum of $f$
are also given as%
\begin{equation}
\operatorname*{esssup}_{\mathbf{x}\in U}f\left(  x\right)  =\sup_{\phi
\in\mathcal{D}\left(  U\right)  ,\phi\geq0,\int\phi=1}\left\langle
\mathsf{f}\left(  \mathbf{x}\right)  ,\phi\left(  \mathbf{x}\right)
\right\rangle \,, \label{SLe 5}%
\end{equation}
and%
\begin{equation}
\operatorname*{essinf}_{\mathbf{x}\in U}f\left(  \mathbf{x}\right)
=\inf_{\phi\in\mathcal{D}\left(  U\right)  ,\phi\geq0,\int\phi=1}\left\langle
\mathsf{f}\left(  x\right)  ,\phi\left(  x\right)  \right\rangle \,.
\label{SLe 6}%
\end{equation}

\end{lemma}

We notice that when $f\in L^{1}\left(  U\right)  $ then (\ref{SLe 5})$\ $could
be $+\infty$ and (\ref{SLe 6})$\ $could be $-\infty.$\smallskip

\section{Comparison of definitions\label{Section: Comparison of definitions}}

We will now study whether the existence of the distributional point value
$f\left(  \mathbf{x}_{0}\right)  $ is equivalent to the existence of $f\left(
\mathbf{x}_{0}\right)  $ \ $\left(  \mathfrak{F}\right)  $ for several
families of delta sequences $\mathfrak{F}.$

\subsection{Standard delta sequences generated by a positive normalized test
function\label{SubS:Standard delta sequences generated by a positive normalized test function}%
}

In this section we consider the family $\mathfrak{F}$ of standard delta
sequences generated by a positive normalized test function of $\mathcal{D}%
\left(  \mathbb{R}^{d}\right)  .$

\begin{proposition}
\label{Prop. DS 2}Let $f\in\mathcal{D}^{\prime}\left(  \mathbb{R}^{d}\right)
.$ Then $f$ has a thick distributional point value at $\mathbf{x}_{0}$ if and
only if for all standard delta sequences generated by a positive normalized
test function of $\mathcal{D}\left(  \mathbb{R}^{d}\right)  ,\ \left\{
\phi_{n}\right\}  _{n=1}^{\infty},$ the limit%
\begin{equation}
\lim_{n\rightarrow\infty}\left\langle f\left(  \mathbf{x}_{0}+\mathbf{x}%
\right)  ,\phi_{n}\left(  \mathbf{x}\right)  \right\rangle =\gamma_{\left\{
\phi_{n}\right\}  }\,, \label{DD 6}%
\end{equation}
exists.
\end{proposition}

\begin{proof}
A standard delta sequences generated by a normalized positive test function
$\phi$ is of the form $\phi_{n}\left(  \mathbf{x}\right)  =n^{d}\phi\left(
n\mathbf{x}\right)  .$ If the distributional thick point value $f_{\mathbf{x}%
_{0}}\left(  \mathbf{w}\right)  =\gamma\left(  \mathbf{w}\right)  $ exists,
$\gamma\in\mathcal{D}^{\prime}\left(  \mathbb{S}\right)  ,$ then
\begin{align*}
\lim_{n\rightarrow\infty}\left\langle f\left(  \mathbf{x}_{0}+\mathbf{x}%
\right)  ,\phi_{n}\left(  \mathbf{x}\right)  \right\rangle  &  =\lim
_{n\rightarrow\infty}\left\langle f\left(  \mathbf{x}_{0}+\mathbf{x}\right)
,n^{d}\phi\left(  n\mathbf{x}\right)  \right\rangle \\
&  =\lim_{n\rightarrow\infty}\left\langle f\left(  \mathbf{x}_{0}+\left(
1/n\right)  \mathbf{x}\right)  ,\phi\left(  \mathbf{x}\right)  \right\rangle
\\
&  =\int_{0}^{\infty}\left\langle \gamma\left(  \mathbf{w}\right)
,\phi\left(  r\mathbf{w}\right)  \right\rangle _{\mathcal{D}^{\prime}\left(
\mathbb{S}\right)  \times\mathcal{D}\left(  \mathbb{S}\right)  }%
r^{d-1}\,\mathrm{d}r\,,
\end{align*}
exists. Conversely, let $\phi$ be a normalized positive test function.\ If the
limit (\ref{DD 6}) exists for all standard delta sequences generated by a
positive normalized test function, it will exist for $\phi_{n}^{\left\{
a\right\}  }\left(  \mathbf{x}\right)  =n^{d}a^{d}\phi\left(  na\mathbf{x}%
\right)  $ for all $a>0.$ Consequently, if the function $\Phi$ is defined as
\begin{equation}
\Phi\left(  a\right)  =\left\langle f\left(  \mathbf{x}_{0}+\mathbf{x}\right)
,a^{d}\phi\left(  a\mathbf{x}\right)  \right\rangle
\,,\ \ \ a>0\,,\label{DD 7}%
\end{equation}
then
\begin{equation}
\lim_{n\rightarrow\infty}\Phi\left(  na\right)  =\lim_{n\rightarrow\infty
}\left\langle f\left(  \mathbf{x}_{0}+\mathbf{x}\right)  ,\phi_{n}^{\left\{
a\right\}  }\left(  \mathbf{x}\right)  \right\rangle =\gamma_{\left\{
\phi_{n}^{\left\{  a\right\}  }\right\}  }\,,\label{DD 8}%
\end{equation}
exists for all $a.$ Since $\Phi$ is continuous, Proposition \ref{Prop. DD 1}
yields that $\gamma_{\left\{  \phi_{n}^{\left\{  a\right\}  }\right\}
}=\gamma_{0}\left(  \phi\right)  $ is independent of $a$ and actually
$\lim_{\lambda\rightarrow\infty}\Phi\left(  \lambda\right)  =\gamma_{0}\left(
\phi\right)  .$ Hence,%
\begin{equation}
\lim_{\varepsilon\rightarrow0^{+}}\left\langle f\left(  \mathbf{x}%
_{0}+\varepsilon\mathbf{x}\right)  ,\phi\left(  \mathbf{x}\right)
\right\rangle =\gamma_{0}\left(  \phi\right)  \,,\label{DD 10}%
\end{equation}
for all normalized positive test functions. Therefore, Lemma \ref{Lemma DD 1}
yields that$\ $the limit $\lim_{\varepsilon\rightarrow0^{+}}\left\langle
f\left(  x_{0}+\varepsilon x\right)  ,\phi\left(  x\right)  \right\rangle
=\gamma_{0}\left(  \phi\right)  $ exists whenever $\phi\in\mathcal{D}\left(
\mathbb{R}^{d}\right)  .$ The formula $\left\langle \gamma_{0},\phi
\right\rangle =\gamma_{0}\left(  \phi\right)  ,$ defines a distribution
$\gamma_{0}\in\mathcal{D}^{\prime}\left(  \mathbb{R}^{d}\right)  ,$ and
$\gamma_{0}$ is homogeneous of degree $0,$ that is, $\gamma_{0}\left(
t\mathbf{x}\right)  =\gamma_{0}\left(  \mathbf{x}\right)  ,$ $t>0.$ As
explained in Section \ref{Section:Preliminaries}, using \cite[Thm.
2.6.2]{EK2002} we conclude that $\gamma_{0}$ is obtained from a distribution
$\gamma\in\mathcal{D}^{\prime}\left(  \mathbb{S}\right)  $ by the formula%
\begin{equation}
\left\langle \gamma_{0},\phi\right\rangle =\int_{0}^{\infty}\left\langle
\alpha\left(  \mathbf{w}\right)  ,\phi\left(  r\mathbf{w}\right)
\right\rangle _{\mathcal{D}^{\prime}\left(  \mathbb{S}\right)  \times
\mathcal{D}\left(  \mathbb{S}\right)  }r^{d-1}\,\mathrm{d}r\,,\label{DD 10b}%
\end{equation}
and that $\alpha$ is the thick distributional value of $f$ at $\mathbf{x}%
_{0}.$\smallskip
\end{proof}

Let $\left\{  \phi_{n}\right\}  _{n=1}^{\infty}$ be a sequence of test
functions. If $T$ is an orthogonal transformation of $\mathbb{R}^{d},$ that
is, with $\left\vert \det T\right\vert =1,$ then the sequence $\left\{
\phi_{n}^{T}\right\}  _{n=1}^{\infty},$ where $\phi^{T}\left(  \mathbf{x}%
\right)  =\phi\left(  T\mathbf{x}\right)  ,$ is also a delta sequence. We have
then the following result.\smallskip

\begin{proposition}
\label{Prop. DS 2a}Let $f\in\mathcal{D}^{\prime}\left(  \mathbb{R}^{d}\right)
.$ Then the distributional point value $f\left(  \mathbf{x}_{0}\right)  $
exists if and only if for all standard delta sequences generated by a positive
normalized test function of $\mathcal{D}\left(  \mathbb{R}^{d}\right)
,\ \left\{  \phi_{n}\right\}  _{n=1}^{\infty},$ the limit $\lim_{n\rightarrow
\infty}\left\langle f\left(  \mathbf{x}_{0}+\mathbf{x}\right)  ,\phi
_{n}\left(  \mathbf{x}\right)  \right\rangle =\gamma_{\left\{  \phi
_{n}\right\}  }\,$exists and for all orthogonal transformations $T$ of
$\mathbb{R}^{d},$ $\gamma_{\left\{  \phi_{n}^{T}\right\}  }=\gamma_{\left\{
\phi_{n}\right\}  }.$

\begin{proof}
This follows immediately from Proposition \ref{Prop. DS 2} if we observe that
a homogeneous function or distribution of degree $0$ is a constant if and only
if it is invariant with respect to orthogonal transformations.\smallskip
\end{proof}
\end{proposition}

Notice that in one variable, Proposition \ref{Prop. DS 2} says that
$\lim_{n\rightarrow\infty}\left\langle f\left(  x_{0}+x\right)  ,\phi
_{n}\left(  x\right)  \right\rangle =\gamma_{\left\{  \phi_{n}\right\}  }$
exists for all standard delta sequences generated by a positive normalized
test function if and only if there are constants $\gamma_{+}$ and $\gamma_{-}$
such that%
\begin{equation}
\lim_{\varepsilon\rightarrow0^{+}}\left\langle f\left(  x_{0}+\varepsilon
x\right)  ,\psi\left(  x\right)  \right\rangle =\gamma_{-}\int_{-\infty}%
^{0}\psi\left(  x\right)  \,\mathrm{d}x+\gamma_{+}\int_{0}^{\infty}\psi\left(
x\right)  \,\mathrm{d}x\,, \label{DD 10c}%
\end{equation}
for all $\psi\in\mathcal{D}\left(  \mathbb{R}\right)  .$ On the other hand,
since the only orthogonal transformations in dimension one are the identity
and $x\rightsquigarrow-x,$ Proposition \ref{Prop. DS 2a}\ says that the
distributional point value $f\left(  x_{0}\right)  $ exists if and only if for
all standard delta sequences generated by a positive normalized test function
of $\mathcal{D}\left(  \mathbb{R}\right)  ,\ \left\{  \phi_{n}\right\}
_{n=1}^{\infty},$ the limit $\lim_{n\rightarrow\infty}\left\langle f\left(
x_{0}+x\right)  ,\phi_{n}\left(  x\right)  \right\rangle =\gamma_{\left\{
\phi_{n}\right\}  }\,$exists and $\gamma_{\left\{  \phi_{n}\left(  -x\right)
\right\}  }=\gamma_{\left\{  \phi_{n}\left(  x\right)  \right\}  }.$

Our results also give the ensuing equivalence.\smallskip

\begin{proposition}
\label{Prop. DS 3}Let $f\in\mathcal{D}^{\prime}\left(  \mathbb{R}^{d}\right)
.$ Then the distributional point value $f\left(  \mathbf{x}_{0}\right)  $
exists and equals $\gamma$ if and only if for $\mathfrak{F}\ $the family of
standard delta sequences generated by a positive normalized test function
\begin{equation}
f\left(  \mathbf{x}_{0}\right)  =\gamma\ \left(  \mathfrak{F}\right)  \,.
\label{DD 10d}%
\end{equation}

\end{proposition}

\subsection{Standard delta sequences generated by an even positive normalized
test
function\label{Subsection:Standard delta sequences generated by an even positive normalized test function}%
}

We now consider the case of symmetric standard delta sequences, the family
considered by Sasane \cite{Sasane}.

We first need to explain the idea of symmetric point values. Let
$f\in\mathcal{D}^{\prime}\left(  \mathbb{R}\right)  $ and $x_{0}\in
\mathbb{R}.$ The symmetric distributional point value of $f$ exists at $x_{0}$
and equals $\gamma$ if%
\begin{equation}
\lim_{\varepsilon\rightarrow0}\frac{f\left(  x_{0}+\varepsilon x\right)
+f\left(  x_{0}-\varepsilon x\right)  }{2}=\gamma\,, \label{Sy 1}%
\end{equation}
in $\mathcal{D}^{\prime}\left(  \mathbb{R}\right)  .$ Each distribution can be
written as the sum of an even one and an odd one,
\begin{equation}
g=g_{\mathrm{e}}+g_{\mathrm{o}}\,, \label{Sy 2}%
\end{equation}
where%
\begin{equation}
g_{\mathrm{e}}\left(  x\right)  =\frac{g\left(  x\right)  +g\left(  -x\right)
}{2}\,,\ \ \ \ g_{\mathrm{o}}\left(  x\right)  =\frac{g\left(  x\right)
-g\left(  -x\right)  }{2}\,. \label{Sy 3}%
\end{equation}
Applying this to $g\left(  x\right)  =f\left(  x_{0}+x\right)  ,$ we see that
the distributional symmetric value $f\left(  x_{0}\right)  $ exists and equals
$\gamma$ if and only if the distributional value $g_{\mathrm{e}}\left(
0\right)  $ exists and equals $\gamma.$

Notice also that if $\phi$ is a test function and we write $\phi
=\phi_{\mathrm{e}}+\phi_{\mathrm{o}},$ then%
\begin{equation}
\left\langle g,\phi\right\rangle =\left\langle g_{\mathrm{e}},\phi
_{\mathrm{e}}\right\rangle +\left\langle g_{\mathrm{o}},\phi_{\mathrm{o}%
}\right\rangle \,. \label{Sy 4}%
\end{equation}
Therefore we have the following result.\smallskip

\begin{lemma}
\label{Lemma DD 1a}A distribution $f\in\mathcal{D}^{\prime}\left(
\mathbb{R}\right)  $ has a symmetric distributional value $\gamma$ at $x_{0}$
if and only if
\begin{equation}
\lim_{\varepsilon\rightarrow0}\left\langle f\left(  x_{0}+\varepsilon
x\right)  ,\phi_{\mathrm{e}}\left(  x\right)  \right\rangle =\gamma
\int_{-\infty}^{\infty}\phi_{\mathrm{e}}\left(  x\right)  \,\mathrm{d}x\,,
\label{Sy 5}%
\end{equation}
for all even test functions $\phi_{\mathrm{e}}.$
\end{lemma}

\begin{proof}
Indeed, if (\ref{Sy 1}) is satisfied, then%
\begin{align*}
\lim_{\varepsilon\rightarrow0}\left\langle f\left(  x_{0}+\varepsilon
x\right)  ,\phi_{\mathrm{e}}\left(  x\right)  \right\rangle  &  =\lim
_{\varepsilon\rightarrow0}\left\langle f\left(  x_{0}+\varepsilon x\right)
-g_{\mathrm{o}}\left(  \varepsilon x\right)  ,\phi_{\mathrm{e}}\left(
x\right)  \right\rangle \\
&  =\lim_{\varepsilon\rightarrow0}\left\langle \frac{f\left(  x_{0}%
+\varepsilon x\right)  +f\left(  x_{0}-\varepsilon x\right)  }{2}%
,\phi_{\mathrm{e}}\left(  x\right)  \right\rangle \\
&  =\gamma\int_{-\infty}^{\infty}\phi_{\mathrm{e}}\left(  x\right)
\,\mathrm{d}x\,.
\end{align*}

Conversely, if (\ref{Sy 5}) holds, then for any test function $\phi
=\phi_{\mathrm{e}}+\phi_{\mathrm{o}},$
\begin{align*}
\lim_{\varepsilon\rightarrow0}\left\langle g_{\mathrm{e}}\left(  \varepsilon
x\right)  ,\phi\left(  x\right)  \right\rangle  &  =\lim_{\varepsilon
\rightarrow0}\left\langle g_{\mathrm{e}}\left(  \varepsilon x\right)
,\phi_{\mathrm{e}}\left(  x\right)  \right\rangle \\
&  =\lim_{\varepsilon\rightarrow0}\left\langle f\left(  x_{0}+\varepsilon
x\right)  ,\phi_{\mathrm{e}}\left(  x\right)  \right\rangle \\
&  =\gamma\int_{-\infty}^{\infty}\phi_{\mathrm{e}}\left(  x\right)
\,\mathrm{d}x\\
&  =\gamma\int_{-\infty}^{\infty}\phi\left(  x\right)  \,\mathrm{d}x\,.
\end{align*}
Hence $g_{\mathrm{e}}\left(  0\right)  =\gamma,$ so that the symmetric
distributional value of $f$ at $x_{0}$ equals $\gamma.$\smallskip
\end{proof}

We can now give an equivalence to the existence of the point value $f\left(
x_{0}\right)  =\gamma$ $(\mathfrak{F}_{\mathrm{sy}}),$ where $\mathfrak{F}%
_{\mathrm{sy}}$ is the family of standard delta sequences generated by a
positive normalized even test function of $\mathcal{D}\left(  \mathbb{R}%
\right)  .$\smallskip

\begin{proposition}
\label{Prop. DD 3}Let $f\in\mathcal{D}^{\prime}\left(  \mathbb{R}\right)  .$
Then the following are equivalent:

1. If $\mathfrak{F}_{\mathrm{sy}}\ $is the family of standard delta sequences
generated by a positive normalized even test function then
\begin{equation}
f\left(  x_{0}\right)  =\gamma\ \ \ \left(  \mathfrak{F}_{\mathrm{sy}}\right)
\,. \label{Sy 6}%
\end{equation}

2. The symmetric distributional point value of $f$ exists at $x_{0}$ and
equals $\gamma.$
\end{proposition}

\begin{proof}
Indeed, if (\ref{Sy 6}) holds then
\begin{equation}
\lim_{n\rightarrow\infty}\left\langle f\left(  x_{0}+x\right)  ,\phi
_{n}\left(  x\right)  \right\rangle =\gamma\,, \label{Sy 7}%
\end{equation}
for all standard delta sequences $\left\{  \phi_{n}\right\}  _{n=1}^{\infty}$
generated by a positive normalized \emph{even} test function $\phi
_{\mathrm{e}},$ and use of Proposition \ref{Prop. DD 1} yields that%
\begin{equation}
\lim_{\varepsilon\rightarrow0}\left\langle f\left(  x_{0}+\varepsilon
x\right)  ,\phi_{\mathrm{e}}\left(  x\right)  \right\rangle =\gamma\,,
\label{DD 13}%
\end{equation}
for such normalized even test functions $\phi_{\mathrm{e}}.$ This last
statement is equivalent to the fact that (\ref{Sy 5}) holds for \emph{all}
even test functions because of Lemma \ref{Lemma DD 1}, and Lemma
\ref{Lemma DD 1a}\ yields that in turn this is equivalent to the symmetric
distributional point value being equal to $\gamma.$\smallskip
\end{proof}

Actually, using the same ideas as in the proof of this Proposition we see that
the limit $\lim_{n\rightarrow\infty}\left\langle f\left(  x_{0}+x\right)
,\phi_{n}\left(  x\right)  \right\rangle =\gamma_{\left\{  \phi_{n}\right\}
}$\ exists for all standard delta sequences $\left\{  \phi_{n}\right\}
_{n=1}^{\infty}$ generated by a positive normalized even test function
$\phi_{\mathrm{e}}$\ if and only if this limit is a constant $\gamma$ and
(\ref{Sy 6}) is satisfied.

\subsection{The family of standard delta sequences generated by a radial
positive normalized test
function\label{Subsection:The family of standard delta sequences generated by a radial positive normalized test function}%
}

We now consider the family $\mathfrak{F}_{\mathrm{rad}}$\ of standard
sequences generated by a radial positive normalized test function.

Let us start with some notation. We denote $r=\left\vert \mathbf{x}\right\vert
$ the radial variable in $\mathbb{R}^{d}.$ A test function $\phi\in
\mathcal{D}\left(  \mathbb{R}^{d}\right)  $ is called \emph{radial} if it is a
function of $r,$ $\phi\left(  \mathbf{x}\right)  =\varphi\left(  r\right)  ,$
for some even function $\varphi\in\mathcal{D}\left(  \mathbb{R}\right)  ;$ the
space of all radial test functions of $\mathcal{D}\left(  \mathbb{R}%
^{d}\right)  $ is denoted as $\mathcal{D}_{\mathrm{rad}}\left(  \mathbb{R}%
^{d}\right)  .$ Similarly, we denote as $\mathcal{D}_{\mathrm{rad}}^{\prime
}\left(  \mathbb{R}^{d}\right)  $ the space of all radial distributions; a
distribution $f\in\mathcal{D}^{\prime}\left(  \mathbb{R}^{d}\right)  $ is
radial if $f\left(  T\mathbf{x}\right)  =f\left(  \mathbf{x}\right)  $ for any
orthogonal transformation of $\mathbb{R}^{d},$ and this actually means
\cite{EstradaRadial, Grafakos-Teschl} that $f\left(  \mathbf{x}\right)
=f_{1}\left(  r\right)  $ for some distribution of one variable $f_{1}.$
Notice, however, that while $\varphi$ is uniquely determined by $\phi,$ for a
given $f$ there are several possible distributions $f_{1}.$

When $d=1$ then $\mathcal{D}_{\mathrm{rad}}\left(  \mathbb{R}\right)  $ and
$\mathcal{D}_{\mathrm{rad}}^{\prime}\left(  \mathbb{R}\right)  $ become the
spaces of even test functions and distributions, respectively, and are also
denoted as $\mathcal{D}_{\mathrm{even}}\left(  \mathbb{R}\right)  $ and
$\mathcal{D}_{\mathrm{even}}^{\prime}\left(  \mathbb{R}\right)  .$ This was
the situation considered in the previous subsection.

Observe that the space $\mathcal{D}_{\mathrm{rad}}^{\prime}\left(
\mathbb{R}^{d}\right)  $ is naturally isomorphic to the dual space $\left(
\mathcal{D}_{\mathrm{rad}}\left(  \mathbb{R}^{d}\right)  \right)  ^{\prime},$
that is to say, if the action of a radial distribution is known in all radial
test functions, then it can be obtained for arbitrary test functions. Indeed,
if $f\in\mathcal{D}_{\mathrm{rad}}^{\prime}\left(  \mathbb{R}^{d}\right)  $
and $\phi\in\mathcal{D}\left(  \mathbb{R}^{d}\right)  ,$ then%
\begin{equation}
\left\langle f\left(  \mathbf{x}\right)  ,\phi\left(  \mathbf{x}\right)
\right\rangle =\left\langle f\left(  \mathbf{x}\right)  ,\widetilde{\phi
}\left(  \mathbf{x}\right)  \right\rangle \,, \label{2.1}%
\end{equation}
where $\widetilde{\phi}\in\mathcal{D}_{\mathrm{rad}}\left(  \mathbb{R}\right)
$ is given as%
\begin{equation}
\widetilde{\phi}\left(  \mathbf{x}\right)  =\phi^{o}\left(  \left\vert
\mathbf{x}\right\vert \right)  \,, \label{2.2}%
\end{equation}
$\phi^{o}\in\mathcal{D}_{\mathrm{even}}\left(  \mathbb{R}\right)  $ being
defined as%
\begin{equation}
\phi^{o}\left(  r\right)  =\frac{1}{\omega}\int_{\mathbb{S}}\phi\left(
r\mathbf{\theta}\right)  \,\mathrm{d}\sigma\left(  \mathbf{\theta}\right)  \,.
\label{2.3}%
\end{equation}
Here we denote by $\mathbb{S}$ the unit sphere of $\mathbb{R}^{d},$
$\mathrm{d}\sigma$ is the Lebesgue measure in $\mathbb{S}$ and $\omega
=2\pi^{d/2}/\Gamma\left(  d/2\right)  $ is the surface area of the sphere.

Equations (\ref{2.2}) and (\ref{2.3})\ define the \emph{radial component of a
test function.} We can also define the radial component of a distribution $f,$
$\widetilde{f}\in\mathcal{D}_{\mathrm{rad}}^{\prime}\left(  \mathbb{R}%
^{d}\right)  ,$ as%
\begin{equation}
\left\langle \widetilde{f}\left(  \mathbf{x}\right)  ,\phi\left(
\mathbf{x}\right)  \right\rangle =\left\langle f\left(  \mathbf{x}\right)
,\widetilde{\phi}\left(  \mathbf{x}\right)  \right\rangle \,.\label{2.4}%
\end{equation}
The distributional analog of (\ref{2.3}) is not well defined, however
\cite{EstradaRadial, Grafakos-Teschl}.

We say that a distribution $f$ has a \emph{radial distributional point value}
at $\mathbf{x}_{0}$ equal to $\gamma$ if%
\begin{equation}
\widetilde{g}\left(  \mathbf{0}\right)  =\gamma\,, \label{2.5}%
\end{equation}
where $\widetilde{g}$ is the radial component of $g\left(  \mathbf{x}\right)
=f\left(  \mathbf{x}_{0}\mathbf{+x}\right)  .$ Similar to Lemma
\ref{Lemma DD 1a}, we have the following characterization.\smallskip

\begin{lemma}
\label{Lemma Rad 1}A distribution $f\in\mathcal{D}^{\prime}\left(
\mathbb{R}^{d}\right)  $ has a radial distributional value $\gamma$ at
$\mathbf{x}_{0}$ if and only if
\begin{equation}
\lim_{\varepsilon\rightarrow0}\left\langle f\left(  \mathbf{x}_{0}%
+\varepsilon\mathbf{x}\right)  ,\phi_{\mathrm{rad}}\left(  \mathbf{x}\right)
\right\rangle =\gamma\int_{\mathbb{R}^{d}}\phi_{\mathrm{rad}}\left(
\mathbf{x}\right)  \,\mathrm{d}\mathbf{x\,}, \label{2.6}%
\end{equation}
for all radial test functions $\phi_{\mathrm{rad}}.$
\end{lemma}

\begin{proof}
If $\widetilde{g}\left(  \mathbf{0}\right)  =\gamma,$ then%
\begin{align*}
\lim_{\varepsilon\rightarrow0}\left\langle f\left(  \mathbf{x}_{0}%
+\varepsilon\mathbf{x}\right)  ,\phi_{\mathrm{rad}}\left(  \mathbf{x}\right)
\right\rangle  &  =\lim_{\varepsilon\rightarrow0}\left\langle \widetilde
{g}\left(  \varepsilon\mathbf{x}\right)  ,\phi_{\mathrm{rad}}\left(
\mathbf{x}\right)  \right\rangle \\
&  =\gamma\int_{\mathbb{R}^{d}}\phi_{\mathrm{rad}}\left(  \mathbf{x}\right)
\,\mathrm{d}\mathbf{x}\,.
\end{align*}

On the other hand, if (\ref{2.6}) holds, then for any test function $\phi,$
\begin{align*}
\lim_{\varepsilon\rightarrow0}\left\langle \widetilde{g}\left(  \varepsilon
\mathbf{x}\right)  ,\phi\left(  \mathbf{x}\right)  \right\rangle  &
=\lim_{\varepsilon\rightarrow0}\left\langle \widetilde{g}\left(
\varepsilon\mathbf{x}\right)  ,\widetilde{\phi}\left(  \mathbf{x}\right)
\right\rangle \\
&  =\lim_{\varepsilon\rightarrow0}\left\langle f\left(  \mathbf{x}%
_{0}+\varepsilon\mathbf{x}\right)  ,\widetilde{\phi}\left(  \mathbf{x}\right)
\right\rangle \\
&  =\gamma\int_{\mathbb{R}^{d}}\widetilde{\phi}\left(  \mathbf{x}\right)
\,\mathrm{d}\mathbf{x}\\
&  =\gamma\int_{\mathbb{R}^{d}}\phi\left(  \mathbf{x}\right)  \,\mathrm{d}%
\mathbf{x}\,.
\end{align*}
Hence $\widetilde{g}\left(  \mathbf{0}\right)  =\gamma,$ that is, the radial
distributional value of $f$ at $\mathbf{x}_{0}$ equals $\gamma.$\smallskip
\end{proof}

Therefore, we obtain the ensuing equivalence for the fact that $f\left(
\mathbf{x}_{0}\right)  =\gamma$ \ $(\mathfrak{F}_{\mathrm{rad}}).$\smallskip

\begin{proposition}
\label{Prop. Rad 1}Let $f\in\mathcal{D}^{\prime}\left(  \mathbb{R}^{d}\right)
.$ Then the following are equivalent:

1. If $\mathfrak{F}_{\mathrm{rad}}\ $is the family of standard delta sequences
generated by a positive normalized radial test function then
\begin{equation}
f\left(  \mathbf{x}_{0}\right)  =\gamma\ \ \ \left(  \mathfrak{F}%
_{\mathrm{rad}}\right)  \,. \label{2.7}%
\end{equation}

2. The radial distributional point value of $f$ exists at $\mathbf{x}_{0}$ and
equals $\gamma.$
\end{proposition}

\begin{proof}
Indeed, if (\ref{2.7}) holds then
\begin{equation}
\lim_{n\rightarrow\infty}\left\langle f\left(  \mathbf{x}_{0}+\varepsilon
\mathbf{x}\right)  ,\phi_{n}\left(  \mathbf{x}\right)  \right\rangle
=\gamma\,, \label{2.8}%
\end{equation}
for all standard delta sequences $\left\{  \phi_{n}\right\}  _{n=1}^{\infty}$
generated by a positive normalized \emph{radial} test function $\phi
_{\mathrm{rad}}.$ Use of Proposition \ref{Prop. DD 1} yields that%
\begin{equation}
\lim_{\varepsilon\rightarrow0}\left\langle f\left(  \mathbf{x}_{0}%
+\varepsilon\mathbf{x}\right)  ,\phi_{\mathrm{rad}}\left(  \mathbf{x}\right)
\right\rangle =\gamma\,, \label{2.9}%
\end{equation}
for such normalized radial test functions $\phi_{\mathrm{rad}}.$ This last
statement is equivalent to the fact that (\ref{Sy 5}) holds for \emph{all}
radial test functions because of Lemma \ref{Lemma DD 1}, and Lemma
\ref{Lemma Rad 1}\ yields that, in turn, this is equivalent to the radial
distributional point value being equal to $\gamma.$\smallskip
\end{proof}

We also have the next result, that is obtained from Lemma \ref{Lemma DD 1}%
.\smallskip

\begin{proposition}
The limit $\lim_{n\rightarrow\infty}\left\langle f\left(  \mathbf{x}%
_{0}+\varepsilon\mathbf{x}\right)  ,\phi_{n}\left(  \mathbf{x}\right)
\right\rangle =\gamma_{\left\{  \phi_{n}\right\}  }$\ exists for all standard
delta sequences $\left\{  \phi_{n}\right\}  _{n=1}^{\infty}$ generated by a
positive normalized radial test function $\phi_{\mathrm{rad}}$\ if and only if
this limit is a constant $\gamma$ and $f\left(  \mathbf{x}_{0}\right)
=\gamma\ \ \left(  \mathfrak{F}_{\mathrm{rad}}\right)  .$
\end{proposition}

\subsection{The family of all positive normalized test
functions\label{Subsection: The family of all positive normalized test functions}%
}

We saw in Subsection
\ref{Subsection:Standard delta sequences generated by an even positive normalized test function}%
\ that Sasane's notion of point values was not equivalent to the standard
definition, nor, in the next subsection, is the notion based on the family
$\mathfrak{F}_{\mathrm{rad}}$\ of standard delta sequences generated by a
positive normalized radial test function. Nevertheless, for the family
$\mathfrak{F}$\ of standard delta sequences generated by a positive normalized
test function the point value definition is in fact equivalent to the standard
\L ojasiewicz\ definition. Of course, Sasane was considering the family of
standard delta sequences generated by an \emph{even }\ positive normalized
test function $\mathfrak{F}_{\mathrm{sy}}.$ Both $\mathfrak{F}_{\mathrm{sy}}%
$\ and $\mathfrak{F}_{\mathrm{rad}}$ are subfamilies of $\mathfrak{F}.$ We can
also consider families larger than $\mathfrak{F}.$ For instance, we can
consider the family $\mathfrak{F}_{\mathrm{all}}$ of \emph{all} delta
sequences formed with positive normalized test functions. In this next
example, we will see that the distributional point value $f\left(
\mathbf{x}_{0}\right)  =\gamma$ is not equivalent to $f\left(  \mathbf{x}%
_{0}\right)  =\gamma\ \ \left(  \mathfrak{F}_{\mathrm{all}}\right)  .$ Later
on we shall find an equivalent formulation to $f\left(  \mathbf{x}_{0}\right)
=\gamma\ \ \left(  \mathfrak{F}_{\mathrm{all}}\right)  .$\smallskip

\begin{example}
\label{Example 3}Let $f$ be the regular distribution given by $f(x)=\sin
(1/x)$. Then $f(0)=0$ distributionally \cite{Lojasiewics}. Let $a_{n}$ be a
positive sequence with $a_{n}\rightarrow0$ and $f(a_{n})=C>0.$ For instance,
we could take $a_{n}=1/(2\pi n+\pi/6).$ For a fixed $n,$ let $\left\{
\psi_{n,m}\right\}  _{m=1}^{\infty}$ be a sequence of positive test functions
such that $\psi_{n,m}\rightarrow\delta(x-a_{n})$ as $m\rightarrow\infty.$ Then
as $n\rightarrow\infty$, we obtain a sequence $\delta_{n}(x)=\delta(x-a_{n})$
that converges to $\delta(x).$ For each $n,$ let $m_{n}$ be large enough so
that
\[
\left\vert \int_{B_{1/n}\left(  a_{n}\right)  }f\left(  x\right)  \psi
_{n,m}\left(  x\right)  \,\mathrm{d}x-f\left(  a_{n}\right)  \right\vert <C/2
\]
and $\operatorname*{supp}\psi_{n,m}\subset B_{1/n}\left(  a_{n}\right)  $ for
$m\geq m_{n}$. Then we can define the sequence $\phi_{n}\left(  x\right)
=\psi_{n,m_{n}}\left(  x\right)  $. By Lemma \ref{Lemma DS 1}, this is a delta
sequence and we have $\langle f(x),\phi_{n}\left(  x\right)  \rangle>C/2$ for
all $n$ and so $\lim\limits_{n\rightarrow\infty}\langle f(x),\phi_{n}\left(
x\right)  \rangle$ cannot be equal to $0.$\smallskip
\end{example}

The next lemma will be useful momentarily.\smallskip

\begin{lemma}
\label{last lemma} If $\left\{  \phi_{n}\right\}  _{n=1}^{\infty}$ is a delta
sequence of positive test functions then
\begin{equation}
\lim_{n\rightarrow\infty}\left\Vert \phi_{n}\right\Vert _{L^{1}\left(
B\setminus U\right)  }=0\,, \label{a}%
\end{equation}
where $B$ and $U$ are both neighborhoods of the origin.
\end{lemma}

\begin{proof}
Choose $\psi\in\mathcal{D}\left(  \mathbb{R}^{d}\right)  $ such that $\psi
\geq0,$%
\begin{equation}
\psi\left(  \mathbf{x}\right)  =1\,,\ \ \ \mathbf{x}\in B\setminus
U\,,\ \ \ \psi\left(  \mathbf{0}\right)  =0\,, \label{b}%
\end{equation}
which is possible because $\mathbf{0}\notin\overline{B\setminus U}.$ Then%
\begin{align*}
\left\Vert \phi_{n}\right\Vert _{L^{1}\left(  B\setminus U\right)  }  &
=\int_{B\setminus U}\left\vert \phi_{n}\left(  \mathbf{x}\right)  \right\vert
\,\mathrm{d}\mathbf{x}=\int_{B\setminus U}\phi_{n}\left(  \mathbf{x}\right)
\,\mathrm{d}\mathbf{x}\\
&  =\int_{B\setminus U}\psi\left(  \mathbf{x}\right)  \phi_{n}\left(
\mathbf{x}\right)  \,\mathrm{d}\mathbf{x}\\
&  \leq\int_{B}\psi\left(  \mathbf{x}\right)  \phi_{n}\left(  \mathbf{x}%
\right)  \,\mathrm{d}\mathbf{x}\rightarrow\psi\left(  \mathbf{0}\right)  =0\,,
\end{align*}
as $n\rightarrow\infty.$\smallskip
\end{proof}

We are now ready to prove the main result of this section.\smallskip

\begin{proposition}
\label{Prop LastS 1}Suppose $f\in\mathcal{D}^{\prime}\left(  \mathbb{R}%
^{d}\right)  $ and $\mathbf{x}_{0}\in\mathbb{R}^{d}.$ If
\begin{equation}
\lim\limits_{n\rightarrow\infty}\langle f\left(  \mathbf{x}_{0}+\mathbf{x}%
\right)  ,\phi_{n}\left(  \mathbf{x}\right)  \rangle=\gamma\,, \label{cond A}%
\end{equation}
for all positive delta sequences $\left\{  \phi_{n}\right\}  $, then the
following two conditions hold:

\begin{enumerate}
\item There is an $r^{\ast}>0$ such that $\left.  f\right\vert _{B_{r^{\ast}%
}\left(  \mathbf{x}_{0}\right)  }\in L^{\infty}\left(  B_{r^{\ast}}\left(
\mathbf{x}_{0}\right)  \right)  $.

\item $\lim\limits_{r\rightarrow0}\left\Vert \left.  f\right\vert
_{B_{r}\left(  \mathbf{x}_{0}\right)  }-\gamma\chi_{B_{r}\left(
\mathbf{x}_{0}\right)  }\right\Vert _{\infty}=0.$
\end{enumerate}

Here $\chi_{B_{r}\left(  \mathbf{x}_{0}\right)  }$ is the characteristic
function of the ball $B_{r}\left(  \mathbf{x}_{0}\right)  .$ Conversely, if
(1) and (2) are satisfied, then (\ref{cond A}) holds for all positive delta
sequences with support contained in $B_{r^{\ast}}\left(  \mathbf{0}\right)  .$
\end{proposition}

\begin{proof}
Suppose that (\ref{cond A}) holds. Notice that (1) follows from Lemma 6.2. To
see that (2) is true, suppose instead that
\begin{equation}
\limsup_{r\rightarrow0}\left\Vert \left.  f\right\vert _{B_{r}\left(
\mathbf{x}_{0}\right)  }-\gamma\chi_{B_{r}\left(  \mathbf{x}_{0}\right)
}\right\Vert _{\infty}=C>0\,. \label{d}%
\end{equation}
Let $r_{n}$ be a decreasing sequence of positive numbers with $r_{n}<r^{\ast}$
and $r_{n}\rightarrow0.$ For each $n$, there is a positive normalized test
function supported in $B_{r_{n}}\left(  \mathbf{0}\right)  ,$ say $\phi_{n},$
such that
\begin{equation}
\left\vert \langle f\left(  \mathbf{x}_{0}+\mathbf{x}\right)  ,\phi_{n}\left(
\mathbf{x}\right)  \rangle-\gamma\right\vert >\frac{C}{2}\,. \label{e}%
\end{equation}
By Lemma \ref{Lemma DS 1}, $\left\{  \phi_{n}\right\}  _{n=1}^{\infty}$ forms
a delta sequence and so $\lim\limits_{n\rightarrow\infty}\langle f\left(
\mathbf{x}_{0}+\mathbf{x}\right)  ,\phi_{n}\left(  \mathbf{x}\right)
\rangle=\gamma,$ which contradicts (\ref{e}).

For the converse, let $\left\{  \psi_{n}\right\}  $ be a delta sequence of
positive normalized test functions supported in $B_{r^{\ast}}\left(
\mathbf{0}\right)  .$ Since by (1) $f$ is a regular distribution in
$B_{r^{\ast}}\left(  \mathbf{0}\right)  $ we have
\begin{equation}
\langle f\left(  \mathbf{x}_{0}+\mathbf{x}\right)  ,\psi_{n}\left(
\mathbf{x}\right)  \rangle-\gamma=\int_{B_{r^{\ast}}\left(  \mathbf{0}\right)
}\left(  f\left(  \mathbf{x}_{0}+\mathbf{x}\right)  -\gamma\right)  \psi
_{n}\left(  \mathbf{x}\right)  \,\mathrm{d}\mathbf{x}\,. \label{f}%
\end{equation}
Let $\varepsilon>0.$ By condition (2), we can find an open neighborhood $V$ of
the origin that is contained in $B_{r^{\ast}}\left(  \mathbf{0}\right)  $ such
that if $W=\mathbf{x}_{0}+V,$\ then $\left\Vert f-\gamma\right\Vert
_{L^{\infty}\left(  W\right)  }<\varepsilon$ and for this $V$\ we can find
$n_{0}$ such that $\left\Vert \psi_{n}\right\Vert _{L_{1}\left(  B_{r^{\ast}%
}\left(  \mathbf{0}\right)  \setminus V\right)  }<\varepsilon$ if $n\geq
n_{0}.$ If $M$ is the constant $\Vert f-\gamma\Vert_{L^{\infty}\left(
B_{r^{\ast}}\left(  \mathbf{x}_{0}\right)  \right)  },$ then we have
\begin{align*}
\left\vert \langle f\left(  \mathbf{x}_{0}+\mathbf{x}\right)  ,\psi_{n}\left(
\mathbf{x}\right)  \rangle-\gamma\right\vert  &  =\left\vert \int_{W}\left(
f\left(  \mathbf{x}\right)  -\gamma\right)  \psi_{n}\left(  \mathbf{x-x}%
_{0}\right)  \,\mathrm{d}\mathbf{x}\right. \\
&  \left.  \ \ \ \ \ \ \ \ \ +\int_{B_{r^{\ast}}\left(  \mathbf{0}\right)
\setminus V}\left(  f\left(  \mathbf{x}_{0}+\mathbf{x}\right)  -\gamma\right)
\psi_{n}\left(  \mathbf{x}\right)  \,\mathrm{d}\mathbf{x}\right\vert \\
&  \leq\left\Vert f-\gamma\right\Vert _{L^{\infty}\left(  W\right)
}\left\Vert \psi_{n}\right\Vert _{L^{1}\left(  V\right)  }\\
&  \ \ \ \ \ \ \ \ \ \ \ +\Vert f-\gamma\Vert_{L^{\infty}\left(  B_{r^{\ast}%
}\left(  \mathbf{x}_{0}\right)  \setminus W\right)  }\left\Vert \psi
_{n}\right\Vert _{L^{1}\left(  B_{r^{\ast}}\left(  \mathbf{0}\right)
\setminus V\right)  }\\
&  <\varepsilon+M\varepsilon\,,
\end{align*}
and consequently $\lim\limits_{n\rightarrow\infty}\langle f\left(
\mathbf{x}_{0}+\mathbf{x}\right)  ,\psi_{n}\left(  \mathbf{x}\right)
\rangle=\gamma.$
\end{proof}

\end{document}